\documentclass[12pt]{amsart}
\usepackage{amsmath,latexsym,amsfonts,amssymb,amsthm}
\usepackage{geometry}
\usepackage{hyperref}
\usepackage{mathrsfs}
\usepackage{graphicx,color}
\usepackage{tikz-cd}
\usepackage{tikz}
\usetikzlibrary{positioning}
\usepackage{mathtools}
\usepackage[all,cmtip]{xy}
\usepackage[alphabetic]{amsrefs}

\newtheorem{prop}{Proposition}[section]
\newtheorem{thm}[prop]{Theorem}
\newtheorem{cor}[prop]{Corollary}

\newtheorem{lem}[prop]{Lemma}

\theoremstyle{definition}

\newtheorem{defn}[prop]{Definition}

\newtheorem{rem}[prop]{\it Remark}

\newtheorem*{claim*}{Claim}

\newcommand{\bP}{\mathbb{P}}
\newcommand{\bC}{\mathbb{C}}
\newcommand{\bR}{\mathbb{R}}

\newcommand{\bQ}{\mathbb{Q}}
\newcommand{\bZ}{\mathbb{Z}}

\newcommand{\cO}{\mathcal{O}}

\newcommand{\cM}{\mathcal{M}}

\newcommand{\cJ}{\mathcal{J}}

\newcommand{\fb}{\mathfrak{b}}

\newcommand{\Supp}{\mathrm{Supp}~}

\newcommand{\mult}{\mathrm{mult}}
\newcommand{\lct}{\mathrm{lct}}

\newcommand{\Pic}{\mathrm{Pic}}
\newcommand{\vol}{\mathrm{vol}}
\newcommand{\ord}{\mathrm{ord}}

\newcommand{\lnlc}{\ell_{\mathrm{nlc}}}
\newcommand{\lnklt}{\ell_{\mathrm{nklt}}}

\begin{document}

\title[Birational superrigidity and K-stability]{Birational superrigidity and K-stability of Fano complete intersections of index one}
\author{Ziquan Zhuang}
\address{Z. Zhuang: Department of Mathematics, Princeton University, Princeton, NJ, 08544-1000.}
\email{zzhuang@math.princeton.edu}
\address{C. Stibitz: Department of Mathematics, Princeton University, Princeton, NJ, 08544-1000.}
\email{cstibitz@math.princeton.edu}
\date{}

\maketitle

\begin{abstract}
    We prove that every smooth Fano complete intersection of index $1$ and codimension $r$ in $\mathbb{P}^{n+r}$ is birationally superrigid and K-stable if $n\ge 10r$. We also propose a generalization of Tian's criterion of K-stability and, as an application, prove the K-stability of the complete intersection of a quadric and a cubic in $\mathbb{P}^5$. In the appendix (written jointly with C. Stibitz), we prove the conditional birational superrigidity of Fano complete intersections of higher index in large dimension.
\end{abstract}

\section{Introduction}

In this paper, we study two different notions on Fano varieties: the \emph{birational superrigidity}, which goes back to the work of \cite{IM-quartic} on quartic threefolds and has been extensively studied in the rationality problem of Fano varieties (see e.g. \cite{P-double-space,dFEM-bounds-on-lct,C-triple-space,dF-hypersurface}); and the \emph{K-stability}, which is closely related to the existence of K\"ahler-Einstein (KE) metric by the celebrated proof from \cite{cds,Tian} of the Yau-Tian-Donaldson conjecture. Despite their theoretical interest, both properties are not so easy to verify in general. Indeed, it is a folklore conjecture that every smooth Fano complete intersection $X\subseteq\bP^n$ is K-polystable, whereas those of index one (i.e. $-K_X$ is linearly equivalent to the hyperplane class) and large dimension are birationally superrigid, and only some partial progress has been made in this direction. For birational superrigidity, the hypersurface case was settled by the work of \cite{IM-quartic,dFEM-bounds-on-lct,dF-hypersurface}; in the case of higher codimensions,  \cite{P-superrigid-cpi-1,P-superrigid-cpi-3,P-superrigid-cpi-2} prove that a \emph{general} member of complete intersections of given degree and codimension is birationally superrigid provided they have large dimension, while \cite{Suzuki-cpi} shows the birational superrigidity of certain families of complete intersections, albeit under some assumptions on the degrees of their defining equations. As for K-stability, the intersection of two (hyper)quadrics is treated in \cite{AGP-two-quadric}, the case of cubic threefolds has been settled recently by \cite{LX-cubic-3fold} and in most remaining cases, we only have the following criterion:

\begin{thm}[\cite{Tian-alpha,OS-alpha,Fujita-alpha}] \label{thm:Tian's criterion}
Let $X$ be a $\bQ$-Fano variety of dimension $n$. Assume that $(X,\frac{n}{n+1}D)$ is log canonical $($resp. klt; or log canonical if $X$ is smooth and $n\ge 2)$ for every effective divisor $D\sim_\bQ -K_X$. Then $X$ is K-semistable $($resp. K-stable$)$.
\end{thm}

This has been successfully applied to smooth hypersurfaces of index $1$ \cite{CP-hypersurface-lct,Fujita-alpha}, to \emph{general} members of some given type of complete intersections of index $1$ \cite{P-KE-cpi,EP-lct-cpi,P-lct-cpi-2} and to certain Fano 3-folds \cite{Noether-Fano,C-3fold-lct,alpha-of-rigid-3fold}. However, the singularities of the pairs as in Theorem \ref{thm:Tian's criterion} can still be hard to control at times, especially for \emph{special} members of a given family. 

The purpose of the present work is therefore twofolds: to introduce another way of proving K-(semi)stability that seems to work well for a large class of Fano varieties without further generality conditions in the corresponding moduli, and to provide a method of estimating log canonical threshold that finds its use in the study of both birational superrigidity and K-stability. As a major application, we prove the following two results.

\begin{thm} \label{thm:superrigidity}
Let $X\subseteq \bP^{n+r}$ be a smooth Fano complete intersection of index $1$, codimension $r$ and dimension $n\ge 10r$. Then $X$ is birationally superrigid.
\end{thm}

\begin{thm} \label{thm:K-stability of cpi}
The following Fano manifolds are K-stable, hence admit KE metric:
    \begin{enumerate}
        \item the complete intersection $X_{2,3}\subseteq\bP^5$ of a quadric and a cubic;
        \item every complete intersection $X\subseteq \bP^{n+r}$ of index $1$, codimension $r$ and dimension $n\ge10r$.
    \end{enumerate}
\end{thm}

Indeed, one can usually further weaken the assumption on the dimension $n$ for each fixed codimension $r$. For example, when $r=2$, we find $n\ge 12$ is enough. 

The fact that the same varieties are involved in both statements is not a mere coincidence and it is actually conjectured \cite{oo-slope-stability,alpha-of-rigid-3fold} that birationally rigid Fano varieties are always K-stable. Although this conjecture is still open, a weaker statement is known.

\begin{thm}[\cite{rigid-imply-stable}] \label{thm:superrigidity implies K-stability}
Let $X$ be a $\bQ$-Fano variety of Picard number $1$. If $X$ is birationally superrigid and $\lct(X;D)\ge \frac{1}{2}$ $($resp. $>\frac{1}{2})$ for every effective divisor $D\sim_\bQ -K_X$, then $X$ is K-semistable $($resp. K-stable$)$.
\end{thm}

This already provides the passage from birational superrigidity to K-stability in many cases, although it does not apply directly to the complete intersection of a quadric and a cubic in $\bP^5$. Indeed, these Fano threefolds are never birationally superrigid and only the general ones are known to be birationally rigid \cite{IP-quadric-cubic}. To verify their K-stability, we interpolate Theorem \ref{thm:superrigidity implies K-stability} with Tian's criterion (Theorem \ref{thm:Tian's criterion}) and further propose the following criterion of K-stability that relates it to some ``weighted" version of birational superrigidity (here a movable boundary is defined as an expression of the form $a\cM$ where $a\in\bQ$ and $\cM$ is a movable linear system; we refer to Section \ref{sec:prelim-rigidity} for more details.)

\begin{thm} \label{thm:criterion}
Let $X$ be a $\bQ$-Fano variety of Picard number $1$ and dimension $n$. Assume that for every effective divisor $D\sim_\bQ -K_X$ and every movable boundary $M\sim_\bQ -K_X$, the pair $(X,\frac{1}{n+1}D+\frac{n-1}{n+1}M)$ is log canonical $($resp. klt$)$. Then $X$ is K-semistable $($resp. K-stable$)$.
\end{thm}

Since $\frac{1}{n+1}D+\frac{n-1}{n+1}M \sim_\bQ -\frac{n}{n+1}K_X$, the assumption above is automatically implied by those of Theorem \ref{thm:Tian's criterion}. However, our assumption seems easier to satisfy as movable boundaries on a Fano variety usually have mild singularities and the most singular divisor $D$ only gets the weight $\frac{1}{n+1}$ (as opposed to $\frac{n}{n+1}$ in Theorem \ref{thm:Tian's criterion}) in our criterion. In particular, if $X$ is birationally superrigid and hence $(X,M)$ has canonical singularities for every movable boundary $M\sim_\bQ -K_X$, then as $\frac{1}{n+1}D+\frac{n-1}{n+1}M$ is a convex combination of $\frac{1}{2}D$ and $M$, we recover Theorem \ref{thm:superrigidity implies K-stability} as a corollary.

By granting Theorem \ref{thm:superrigidity} and \ref{thm:superrigidity implies K-stability}, the second part of Theorem \ref{thm:K-stability of cpi} is reduced to an estimate of log canonical thresholds on the varieties in question. An amusing fact is that in our case, this latter problem turns out to be almost identical to proving birational superrigidity itself, so in some sense we get both birational superrigidity and K-stability for free once we know how to provide the required lower bound of log canonical thresholds. A key ingredient for such estimate is given by the following.

\begin{thm} \label{thm:lct-estimate}
Let $(X,\Delta)$ be a pair. Let $D$ be an effective $\bQ$-divisor on $X$ and $L$ a line bundle. Let $\lambda>0$ be a constant. Assume the following:
    \begin{enumerate}
        \item $L-(K_X+\Delta+(1-\epsilon)D)$ is nef and big and $(X,\Delta+(1-\epsilon)D)$ is klt outside a finite set $T$ of points for all $0<\epsilon\ll 1$;
        \item for all $0$-dimensional subschemes $\Sigma\subseteq X$ supported on $T$ such that $\ell(\cO_\Sigma)\le h^0(X,L)$, we have $\lct(X,\Delta;\Sigma)\ge\lambda$.
    \end{enumerate}
Then $\lct(X,\Delta;D)\ge\frac{\lambda}{\lambda+1}$. Moreover, when equality holds, there exists some $0$-dimensional subscheme $\Sigma\subseteq X$ satisfying the assumption $(2)$ such that every divisor that computes $\lct(X,\Delta;D)$ also computes $\lct(X,\Delta;\Sigma)=\lambda$.
\end{thm}

For example, if $h^0(X,L)=0$, then we may choose $\lambda$ to be any constant, and the theorem implies that $\lct(X,\Delta;D)\ge1$. As another example, if $X$ is smooth, $\Delta=0$ and $L$ is the trivial line bundle, then $\lambda=n=\dim X$ satisfies the assumption (2) and we have $\lct(X;D)\ge\frac{n}{n+1}$, with equality if and only if $\mult_x(D)=n+1$ for some $x\in X$ (since $\lct(X;x)$ is computed exactly by the blowup of $x$ in this case). These observations lead to a simple proof of the following well-known result.

\begin{cor}[\cite{CP-hypersurface-lct,dFEM-bounds-on-lct}] \label{cor:hypersurface lct}
Let $X\subseteq\bP^{n+1}$ $(n\ge3)$ be a smooth hypersurface of degree $d$ and let $H$ be the hyperplane class. Then
    \begin{enumerate}
        \item $\lct(X;|H|_\bQ)=1$ if $d\le n$,
        \item $\lct(X;|H|_\bQ)\ge\frac{n}{d}$ if $d\ge n+1$.
    \end{enumerate}
In the latter case, equality holds if and only if $X$ has an Eckardt point $($i.e., there exists a hyperplane section with multiplicity $d$ at the point$)$.
\end{cor}

More interesting applications come in when we apply Theorem \ref{thm:lct-estimate} to the case when $L$ has some positivity. Indeed, the proofs of most results in this article consist of multiple uses of Theorem \ref{thm:lct-estimate} in this setting. In those cases, we can usually find the constant $\lambda$ by the work of \cite{dFEM-mult-and-lct} (or its variants) and this in particular yields the following.

\begin{cor} \label{cor:lct>1/2}
Let $X,D,L$ be as in Theorem \ref{thm:lct-estimate} and let $\Delta=0$. Assume that $X$ is smooth and $h^0(X,L)\le \frac{n^n}{n!}$. Then $\lct(X;D)>\frac{1}{2}$.
\end{cor}

As will be clear from the proof, the number $\frac{n^n}{n!}$ can be replaced by the minimum number of lattice points in the simplex $Q_{\mathbf{a}}=\{\mathbf{x}\in\bR^n_{\ge0}\,|\,\mathbf{a}\cdot\mathbf{x} < 1\}$ among all possible choices of $\mathbf{a}\in\bR^n_+$ such that $(1,1,\cdots,1)\in\overline{Q_{\mathbf{a}}}$. More generally, if $X$ is singular, we may replace it by the minimal non-klt colengths (see Section \ref{sec:lct estimate}) of the singularities. These observations will be important in the appendix where we study complete intersections of higher index and in the forthcoming work \cite{LZ-singular-cpi} when we consider singular complete intersections.

Apart from its obvious connection to Theorem \ref{thm:superrigidity implies K-stability}, Corollary \ref{cor:lct>1/2} is also a key step in the proof of Theorem \ref{thm:superrigidity}. The idea is that, given a movable boundary $M\sim_\bQ -K_X$ on a complete intersection $X$ of index $1$, one can usually show that $(X,2M)$ is log canonical outside a set of small dimension (as in  \cite{dF-hypersurface}), so after cutting down by hyperplanes, we can always reduce to the setting of Theorem \ref{thm:lct-estimate} and it suffices to show that $\lct(X;2M)\ge \frac{1}{2}$. In the hypersurface case \cite{dFEM-bounds-on-lct,dF-hypersurface}, this is done by projecting $X$ to $\bP^n$ and then applying \cite{dFEM-bounds-on-lct,dFEM-mult-and-lct}. Such strategy does not seem to carry over to complete intersections since the projection is in general a hypersurface of large degree (compared to the dimension) and the bound on log canonical threshold given by the argument of \cite{dFEM-bounds-on-lct,dF-hypersurface} is not sufficient. To get around this issue, we estimate the log canonical threshold using local information of the multiplier ideal (which is also used in the argument of \cite{dF-hypersurface}) in a way that does not require taking projections, and Corollary \ref{cor:lct>1/2} plays an important role here.

This paper is organized as follows. In Section \ref{sec:prelim}, we collect some important definitions and results that are used throughout the paper. Theorem \ref{thm:lct-estimate} is proved in Section \ref{sec:lct estimate}, where we also apply it to many different problems, proving along the way Corollaries \ref{cor:hypersurface lct} and \ref{cor:lct>1/2} and Theorems \ref{thm:superrigidity} and \ref{thm:K-stability of cpi}(2). In Section \ref{sec:criterion}, we prove the K-(semi)stability criterion, Theorem \ref{thm:criterion}, and then apply it to prove the K-stability of $X_{2,3}\subseteq\bP^5$ in Section \ref{sec:X_2,3}. In the appendix (written jointly with C. Stibitz), we consider Fano complete intersections of higher index and prove their conditional birational superrigidity.

\subsection*{Acknowledgement}

The author would like to thank his advisor J\'anos Koll\'ar for constant support, encouragement and numerous inspiring conversations. He also extends his thanks to Yuchen Liu, Xiaowei Wang and Chenyang Xu for several interesting conversations around K-stability; to Simon Donaldson for his interest and comments; to Weibo Fu, Lue Pan, Charlie Stibitz and Fan Zheng for helpful discussions; and to the anonymous referee(s) for careful reading of the manuscript.

\section{Preliminary} \label{sec:prelim}

\subsection{Notation and conventions}

We work over the field $\bC$ of complex numbers throughout the paper. Unless otherwise specified, all varieties are assumed to be projective and normal, and divisors are understood as $\bQ$-divisors. A \emph{pair} $(X,D)$ consists of a variety $X$ and an effective divisor $D\subseteq X$ such that $K_X+D$ is $\bQ$-Cartier. The notions of terminal, canonical, klt and log canonical (lc) singularities are defined in the sense of \cite[Definition 2.8]{mmp}. A variety $X$ is said to be $\bQ$-Fano if $-K_X$ is $\bQ$-Cartier and ample and $X$ has klt singularities. Let $(X,\Delta)$ be a pair and $D$ a $\bQ$-Cartier divisor on $X$, the \emph{log canonical threshold}, denoted by $\lct(X,\Delta;D)$ (or simply $\lct(X;D)$ when $\Delta=0$), of $D$ with respect to $(X,\Delta)$ is the largest number $t$ such that $(X,\Delta+tD)$ is log canonical. Similarly, the notation $\lct(X,\Delta;|D|_\bQ)$ (so-called global log canonical threshold) stands for the infimum of $\lct(X,\Delta;D')$ among all effective divisors $D'\sim_\bQ D$ while $\lct(X,\Delta;Z)$ refers to the log canonical threshold of a subscheme $Z\subseteq X$.

\subsection{K-stability}

We refer to \cite{Tian-K-stability-defn,Don-K-stability-defn} for the original definition of K-stability using test configurations. In the rest of this article we use the following equivalent valuative criterion.

\begin{defn}[{\cite[Definition 1.1]{Fujita-valuative-criterion}}] \label{defn:threshold and beta}
Let $X$ be a variety of dimension $n$ and $L$ an ample divisor on $X$. Let $F$ be a prime divisor over $X$, i.e., there exists a projective birational morphism $\pi: Y\to X$ with $Y$ normal such that $F$ is a prime divisor on $Y$.
    \begin{enumerate}
        \item For any $x\ge 0$, we define $\vol_X(L-xF):=\vol_Y(\pi^*L-xF)$.
        \item The \emph{pseudo-effective threshold} $\tau(L,F)$ (or simply $\tau(F)$ when the choice of $L$ is clear) of $L$ with respect to $F$ is defined as
        \[\tau(F):=\sup\{\tau>0\,|\, \vol_X(L-\tau F)>0\}.\]
        \item Let $A_X(F)$ be the log discrepancy of $F$ with respect to $X$. We set
        \[\beta(F):=A_X(F)\cdot(L^n)-\int_0^{\infty}\vol_X(L-xF)\mathrm{d}x.\]
        \item $F$ is said to be \emph{dreamy} (with respect to $L$) if the graded algebra
        \[\bigoplus_{k,j\in\bZ_{\geq 0}}H^0(Y, kr\pi^*L-jF)\]
        is finitely generated for some (hence, for any) $r\in\bZ{>0}$ with $rL$ Cartier. 
    \end{enumerate}
All the above definitions do not depend on the choice of the morphism $\pi: Y\to X$ (they only depend on the divisorial valuation on the function field of $X$ given by $F$). When $X$ is a $\bQ$-Fano variety, we define the corresponding $\tau(F)$ and $\beta(F)$ by taking $L=-K_X$.
\end{defn}

\begin{thm}[{\cite[Theorems 1.3 and 1.4]{Fujita-valuative-criterion} and \cite[Theorem 3.7]{Li-equivariant-minimize}}] \label{thm:beta-criterion}
Let $X$ be a $\bQ$-Fano variety. Then $X$ is K-stable $($resp.\ K-semistable$)$ if and only if $\beta(F)>0$ $($resp.\ $\beta(F)\geq 0)$ holds for any dreamy prime divisor $F$ over $X$. 
\end{thm}

\subsection{Birational superrigidity} \label{sec:prelim-rigidity}

A Fano variety $X$ is said to be birationally superrigid if it has terminal singularities, it is $\bQ$-factorial of Picard number one and every birational map $f:X\dashrightarrow Y$ from $X$ to a Mori fiber space is an isomorphism (see e.g. \cite[Definition 1.25]{Noether-Fano}). In particular, birationally superrigid Fano varieties are not rational. For this paper, the following equivalent characterization using maximal singularities (sometimes also referred to as the Noether-Fano inequality) is more useful.

\begin{defn} \label{defn:movable boundary}
Let $(X,D)$ be a pair. A movable boundary on $X$ is defined as an expression of the form $a\cM$, where $a\in\bQ$ and $\cM$ is a movable linear system on $X$. Its $\bQ$-linear equivalence class is defined in an evident way. If $M=a\cM$ is a movable boundary, we say that the pair $(X,D+M)$ is klt (resp. canonical, lc) if for $k\gg 0$ and for general members $D_1,\cdots,D_k$ of the linear system $\cM$, the pair $(X,D+M_k)$ (where $M_k=\frac{a}{k}\sum_{i=1}^k D_i$) is klt (resp. canonical, lc) in the usual sense (alternatively, it can also be defined via the singularity type of $(X,D;\fb^a)$ where $\fb$ is the base ideal of $\cM$). For simplicity, we usually do not distinguish between the movable boundary $M$ and the actual divisor $M_k$ for suitable $k$.
\end{defn}

\begin{thm}[{\cite[Theorem 1.26]{Noether-Fano}}]
Let $X$ be a Fano variety. Then it is birationally superrigid if and only if it has $\bQ$-factorial terminal singularities, it has Picard number one, and for every movable boundary $M\sim_\bQ -K_X$ on $X$, the pair $(X,M)$ has canonical singularities.
\end{thm}

\section{Lower bounds of log canonical thresholds} \label{sec:lct estimate}

In this section we prove Theorem \ref{thm:lct-estimate} and its applications.

\begin{proof}[Proof of Theorem \ref{thm:lct-estimate}]
We may assume that $\lct(X,\Delta;D)<1$; otherwise there is nothing to prove. Let $0<\epsilon\ll 1$. By the first assumption, the multiplier ideal $\cJ=\cJ(X,\Delta+(1-\epsilon)D)$ defines a $0$-dimensional subscheme $\Sigma\subseteq X$ supported on $T$ that does not depend on $\epsilon$. Since $L-(K_X+\Delta+(1-\epsilon)D)$ is nef and big, by Nadel vanishing we have $H^1(X,\cJ(X,\Delta+(1-\epsilon)D)\otimes L)=0$; thus the natural restriction map $H^0(X,L)\to H^0(\Sigma, L|_\Sigma)\cong H^0(\Sigma, \cO_\Sigma)$ is surjective. In particular, $\ell(\cO_\Sigma)\le h^0(X,L)$. Hence, by our second assumption, $\lct(X,\Delta;\Sigma)\ge\lambda$. Now let $E$ be a divisor over $X$ that computes $\lct(X,\Delta;D)$. By the definition of the multiplier ideal, for every $f\in\cJ$  we have 
\[\ord_E(f)\ge \lfloor (1-\epsilon)\ord_E(D)-A_{(X,\Delta)}(E)+1 \rfloor > (1-\epsilon)\ord_E(D)-A_{(X,\Delta)}(E)
\]
where $A_{(X,\Delta)}(E)$ is the log discrepancy of $E$ with respect to $(X,\Delta)$. Letting $\epsilon\rightarrow 0$ we get
\begin{equation} \label{ineq:ord_E}
    \ord_E(f)\ge \ord_E(D)-A_{(X,\Delta)}(E).
\end{equation}
On the other hand, if $f$ is general in $\cJ$, then we have
\begin{equation} \label{ineq:lct_J}
    \frac{A_{(X,\Delta)}(E)}{\ord_E(f)}\ge \lct(X,\Delta;\Sigma)\ge \lambda.
\end{equation}
Combining these two inequalities we obtain $\lambda^{-1}A_{(X,\Delta)}(E)\ge \ord_E(D)-A_{(X,\Delta)}(E)$, which reduces to $\lct(X,\Delta;D)=\frac{A_{(X,\Delta)}(E)}{\ord_E(D)}\ge \frac{\lambda}{\lambda+1}$. If equality holds, then the inequality \eqref{ineq:lct_J} is an equality, hence, in particular, we have $\lct(X,\Delta;\Sigma)=\lambda$, and it is computed by $E$.
\end{proof}

\begin{rem} \label{rem:pointwise-estimate}
Using the same argument we can also get a pointwise statement as follows. Keeping notation from the above proof, let $\Sigma=\cup_{i=1}^r \Sigma_i$ be the decomposition of $\Sigma$ into connected components and let $x_i=\Supp(\Sigma_i)$. Then we have that $\lct(X,\Delta;\Sigma_i)\ge\lambda$ implies $\lct(X,\Delta;D)\ge\frac{\lambda}{\lambda+1}$ in a neighborhood of $x_i$, with equality if and only if every exceptional divisor centered at $x_i$ that computes $\lct(X,\Delta;D)$ around $x_i$ also computes $\lct(X,\Delta;\Sigma_i)$. This observation will be important in the proof below as well as in the last section.
\end{rem}

Let us first apply Theorem \ref{thm:lct-estimate} to compute log canonical thresholds on hypersurfaces. 

\begin{proof}[Proof of Corollary \ref{cor:hypersurface lct}]
Let $D\sim_\bQ H$ be an effective divisor on $X$. It suffices to show that $\lct(X;D)\ge\min\{\frac{n}{d},1\}$. By \cite[Proposition 5]{P-rigid-hypersurface}, $\mult_x D\le 1$ except at finitely many points $x\in X$, hence by \cite[(3.14.1)]{Kol-sing-of-pairs}, $(X,(1-\epsilon)D)$ is klt outside a finite set of points ($0<\epsilon\ll 1$). If $d\le n$, then we may apply Theorem \ref{thm:lct-estimate} with $L=\cO_X(-H)$, $\Delta=0$ and obtain $\lct(X;D)\ge \frac{\lambda}{\lambda+1}$ for any $\lambda>0$, thus $\lct(X;D)\ge 1$. If $d\ge n+1$, then let $x\in X$ and let $\gamma:X\rightarrow \bP^n$ be a general linear projection such that $\gamma$ is \'etale in the neighbourhood of $x$ and $\gamma|_D$ is injective in the neighbourhood of $\gamma(x)$. We then have $\lct(X;D)=\lct(\bP^n,\gamma(D))$ near $x$, and since $\frac{n+1}{d}\le 1$, $(\bP^n,\frac{n+1}{d}(1-\epsilon)\gamma(D))$ is klt in a punctured neighbourhood of $\gamma(x)$. We apply Theorem \ref{thm:lct-estimate} to the pair $(\bP^n,\frac{n+1}{d}\gamma(D))$ with $L=0\sim_\bQ K_{\bP^n}+\frac{n+1}{d}\gamma(D)$, $\Delta=0$, and $T=\{\gamma(x)\}$. Note that the only $0$-dimensional subscheme $\Sigma$ supported at $\gamma(x)$ with $\ell(\cO_\Sigma)\le h^0(\bP^n,\cO_{\bP^n})=1$ is the closed point $\gamma(x)$ itself, and for such a point we always have $\lct(\bP^n;\gamma(x))=n$. Hence we may take $\lambda=n$ and obtain $\lct(\bP^n;\frac{n+1}{d}\gamma(D))\ge \frac{n}{n+1}$. It follows that $(\bP^n,\frac{n}{d}\gamma(D))$ is log canonical at $\gamma(x)$ and hence $(X,\frac{n}{d}D)$ is log canonical at $x$ as well. Since $x\in X$ is arbitrary, we get $\lct(X;D)\ge \frac{n}{d}$. Suppose that equality $\lct(X;|H|_\bQ)=\frac{n}{d}$ holds, then by \cite[Theorem 1.5]{Birkar}, there exists $D\sim_\bQ -K_X$ (which we may assume to be irreducible as $X$ has Picard number one) with $\lct(X;D)=\frac{n}{d}$. Let $x\in X$ be a point where $(X,\frac{n}{d}D)$ is not klt and let $\gamma:X\rightarrow\bP^n$ be as before. Then by the equality case of Theorem \ref{thm:lct-estimate}, every divisor that computes $\lct(\bP^n;\gamma(D))$ also computes $\lct(\bP^n;\gamma(x))$. It follows that $\lct(\bP^n;\gamma(D))$ is computed by $\mult_{\gamma(x)}$ and hence $\mult_x D=\mult_{\gamma(x)} \gamma(D)=d$. If $D$ is not a hyperplane section, let $W=T_x X\cap X$ be the restriction of the tangent hyperplane at $x$, we have $\mult_x W\ge 2$ and $d=\deg(D\cdot W)\ge\mult_x(D\cdot W)\ge 2d$, a contradiction. Hence $D$ is a hyperplane section with multiplicity $d$ at $x$.
\end{proof}

Next we use Theorem \ref{thm:lct-estimate} to give some lower bounds of log canonical thresholds on complete intersections. To this end, we introduce the following definition.

\begin{defn}
Let $x\in (X,D)$ be a klt singularity. The minimal non-klt (resp. non-lc) colength of $x\in (X,D)$ with coefficient $\lambda$ is defined as
 \begin{align*}
    \lnklt(x,X,D;\lambda) & :=   \min\{\ell(\cO_X/\cJ)\,|\,\mathrm{Supp}(\cO_X/\cJ)=\{x\}\:\mathrm{and}\:(X,D;\cJ^\lambda)\:\mathrm{is}\:\mathrm{not}\:\mathrm{klt}\} \\
    (\mathrm{resp.}\;\lnlc(x,X,D;\lambda) & :=  \min\{\ell(\cO_X/\cJ)\,|\,\mathrm{Supp}(\cO_X/\cJ)=\{x\}\:\mathrm{and}\:(X,D;\cJ^\lambda)\:\mathrm{is}\:\mathrm{not}\:\mathrm{lc}\}).
 \end{align*}
When $D=0$, we use the abbreviation $\lnklt(x,X;\lambda)$ (resp. $\lnlc(x,X;\lambda)$).
\end{defn}

We can then rephrase Theorem \ref{thm:lct-estimate} in terms of minimal non-klt (resp. non-lc) colengths.

\begin{thm} \label{thm:colength and lct}
Let $(X,\Delta)$ be a klt pair, let $D$ an effective divisor on $X$ and let $L$ a line bundle such that $L-(K_X+\Delta+(1-\epsilon)D)$ is big and nef for $0<\epsilon\ll 1$. Assume that $(X,\Delta+D)$ is log canonical outside a finite set of points $T$ and that $h^0(X,L) < \lnklt(x,X,\Delta;\lambda)$ $($resp. $< \lnlc(x,X,\Delta;\lambda))$ for every $x\in T$. Then $\lct(X,\Delta;D)>\frac{\lambda}{\lambda+1}$ $($resp. $\ge \frac{\lambda}{\lambda+1})$.  \qed
\end{thm}

Hence for various applications, it suffices to find a suitable lower bound of the minimal non-klt (resp. non-lc) colengths and compare it with $h^0(X,L)$. In the smooth case, this can be given by the work of \cite{dFEM-mult-and-lct} (or more precisely, by the proof therein). To state the result, we need more notation: for $\mathbf{a}\in\bR^n_+$ and $\lambda>0$, let
\begin{align*}
    Q_{\mathbf{a}} & =  \{\mathbf{x}=(x_1,\cdots,x_n)\in\bR^n\,|\,x_1\ge 0,\cdots,x_n\ge0,\mathbf{a}\cdot\mathbf{x} < 1\}, \\
    \sigma_{n,\lambda} & =  \min\{\#(Q_\mathbf{a}\cap \bZ^n)\,|\,\mathbf{a}\in\bR^n_+\;\mathrm{s.t.}\;(\lambda,\lambda,\cdots,\lambda)\in Q_\mathbf{a}\}, \\
    \bar{\sigma}_{n,\lambda} & =  \min\{\#(Q_\mathbf{a}\cap \bZ^n)\,|\,\mathbf{a}\in\bR^n_+\;\mathrm{s.t.}\;(\lambda,\lambda,\cdots,\lambda)\in\overline{Q_\mathbf{a}}\}.
\end{align*}
Clearly $\sigma_{n,\lambda}\ge\bar{\sigma}_{n,\lambda}$.


\begin{lem} \label{lem:sigma_n and lct}
Let $X$ be a smooth variety of dimension $n$ and $x\in X$. Let $\lambda>0$. Then 
\[\lnlc(x,X;\lambda^{-1})\ge \sigma_{n,\lambda}, \quad \lnklt(x,X;\lambda^{-1})\ge \bar{\sigma}_{n,\lambda}.
\]
\end{lem}

\begin{proof}
We may assume that $(X,x)=(\mathbb{A}^n,0)$ since the statement is \'etale local. Moreover, as in the proof of \cite[Theorem 1.1]{dFEM-mult-and-lct}, we may assume that $\cJ\subseteq\cO_X$ is a monomial ideal by the lower semicontinuity (see e.g. \cite{DK-semicontinuity-lct}) of log canonical thresholds. Let $P$ be the Newton polytope of $\cJ$, defined as the convex hull in $\bR^n_{\ge0}$ of all the points corresponding to monomials in $\cJ$. By \cite{Howald}, letting $\mu=\lct(\mathbb{A}^n;\cJ)^{-1}$, we have
\[\mu=\min\{t>0\,|\,(t,t,\cdots,t)\in P\}.
\]
Let $W$ be a supporting hyperplane of $P$ at $(\mu,\cdots,\mu)\in \partial P$. Write the equation of $W$ as $\mathbf{a}\cdot \mathbf{x}=1$ where $\mathbf{a}\in\bR^n_+$, then we have $\ell(\cO_X/\cJ)=\#((\bR^n_{\ge0}\backslash P)\cap \bZ^n)\ge \#(Q_\mathbf{a}\cap \bZ^n)$. If $(X;\cJ^{1/\lambda})$ is not lc (resp. not klt), then $\mu>\lambda$ (resp. $\ge\lambda$), hence $(\lambda,\cdots,\lambda)\in Q_\mathbf{a}$ (resp. $\in\overline{Q_\mathbf{a}}$) and the lemma simply follows from the definition of $\sigma_{n,\lambda}$ (resp. $\bar{\sigma}_{n,\lambda}$).
\end{proof}

\begin{cor} \label{cor:lct>1/(1+lambda)}
Let $X,D,L$ be as in Theorem \ref{thm:lct-estimate} and $\Delta=0$. Assume that $X$ is smooth of dimension $n$ and $h^0(X,L)<\bar{\sigma}_{n,\lambda}$ $($resp. $<\sigma_{n,\lambda})$. Then $\lct(X;D)>\frac{1}{\lambda+1}$ $($resp. $\ge\frac{1}{\lambda+1})$.
\end{cor}

\begin{proof}
This is immediate from Theorem \ref{thm:colength and lct} and Lemma \ref{lem:sigma_n and lct}.
\end{proof}

In light of this, all subsequent estimates of log canonical thresholds essentially reduce to finding lower bounds of $\sigma_{n,\lambda}$ (or $\bar{\sigma}_{n,\lambda}$). Here are some sample applications:

\begin{proof}[Proof of Corollary \ref{cor:lct>1/2}]
It is clear that $\bar{\sigma}_{n,1}>\vol(Q_\mathbf{a})\ge \frac{n^n}{n!}$ if $(1,\cdots,1)\in \overline{Q_\mathbf{a}}$ and $n\ge2$ (see the proof of \cite[Theorem 1.1]{dFEM-mult-and-lct}), so the result follows directly from Corollary \ref{cor:lct>1/(1+lambda)} with $\lambda=1$.
\end{proof}

\begin{lem} \label{lem:lct-cpi}
Let $X\subseteq\bP^{n+r}$ be a smooth Fano complete intersection of codimension $r$ and dimension $n\ge 6r$. Let $H$ be the hyperplane class. Then $\lct(X;|H|_\bQ)>\frac{1}{2}$.
\end{lem}

\begin{proof}
By \cite[Theorem 1.5]{Birkar}, it suffices to show that for every $D\sim_\bQ H$ we have $\lct(X;D)>\frac{1}{2}$. By \cite[Proposition 2.1]{Suzuki-cpi}, we have $\mult_S(D)\le 1$ for every subvariety $S\subseteq X$ of dimension $r$, hence for all $0<\epsilon\ll 1$, the pair $(X,(1-\epsilon)D)$ is klt outside a subset of dimension at most $r-1$ in $X$. Let $x\in X$ be an arbitrary point and let $Y=X\cap V\subseteq\bP^{n+1}$ be a general linear space section containing $x$ of codimension $r-1$. Let $D_Y=D|_Y$ and $L=(r-1)H|_Y$. Then by adjunction $L-(K_Y+(1-\epsilon)D_Y)$ is ample and the pair $(Y,(1-\epsilon)D_Y)$ is klt outside a finite set of points. Since
\[
    h^0(Y,L)\le h^0(\bP^{n+1},\cO_{\bP^{n+1}}(r-1))=\binom{n+r}{r-1}
\]
always holds, we have $\lct(Y;D_Y) > \frac{1}{2}$ by Corollary \ref{cor:lct>1/2} as long as
\begin{equation} \label{eq:h^0(L)-lct}
    \binom{n+r}{r-1} \le \frac{(n-r+1)^{n-r+1}}{(n-r+1)!}.
\end{equation}
Granting this for the moment, then $(Y,\frac{1}{2}D_Y)$ is klt and by inversion of adjunction (see e.g. \cite[Theorem 4.9]{mmp}) $(X,Y+\frac{1}{2}D)$ is plt in a neighbourhood of $Y$. In particular, $(X,\frac{1}{2}D)$ is klt at $x$. Since $x$ is arbitrary, we see that $(X,\frac{1}{2}D)$ is klt.

It remains to prove \eqref{eq:h^0(L)-lct} when $n\ge 6r$. As $\frac{r^r}{r!}<e^r$, we see that $\binom{n+r}{r-1}<\binom{n+r}{r}\le\frac{(n+r)^r}{r!}< e^r (a+1)^r$ where $a=\frac{n}{r}$; on the other hand, $\frac{(n-r+1)^{n-r+1}}{(n-r+1)!}>\frac{(n-r)^{n-r}}{(n-r)!}>2^{n-r}=2^{(a-1)r}$ when $n-r\ge 6$ (we may assume that $r\ge2$ by Corollary \ref{cor:hypersurface lct}), so \eqref{eq:h^0(L)-lct} holds as long as $2^{a-1}\ge e(a+1)$, which is trivial since $a\ge 6$.
\end{proof}

It is not hard to see that one can actually do slightly better if a more precise value of $\sigma_{n,\lambda}$ or $\bar{\sigma}_{n,\lambda}$ is known. For example, we have the following.

\begin{lem} \label{lem:lct-codim2-cpi}
Let $X\subseteq\bP^{n+2}$ be a smooth Fano complete intersection of codimension $2$ and dimension $n\ge 4$. Let $H$ be the hyperplane class. Then $\lct(X;|H|_\bQ)>\frac{1}{2}$.
\end{lem}

\begin{proof}
Taking $r=2$ in \eqref{eq:h^0(L)-lct} and using Lemma \ref{cor:lct>1/(1+lambda)} instead of Corollary \ref{cor:lct>1/2} in the proof of Lemma \ref{lem:lct-cpi}, we see that is suffices to show that $n+2=\binom{n+r}{r-1}<\bar{\sigma}_{n-1,1}$. For $n\ge 4$, this follows from the next lemma.
\end{proof}

\begin{lem}
$\bar{\sigma}_{n,1}\ge 2^n-1$.
\end{lem}

\begin{proof}
Let $Q_\mathbf{a}$ be such that $\mathbf{e}=(1,\cdots,1)\in \overline{Q_\mathbf{a}}$. It suffices to show that every vertex (other than $\mathbf{e}$) of the unit cube $[0,1]^n$ is contained in $Q_\mathbf{a}$. But if $v$ is such a vertex, then as $\mathbf{a}\in\bR^n_+$ we have $1\ge \mathbf{a}\cdot\mathbf{e}>\mathbf{a}\cdot\mathbf{v}$ as desired.
\end{proof}

\begin{cor} \label{cor:X_2,4}
The smooth complete intersection in $\bP^6$ of a quadric and a quartic not containing a plane is K-stable.
\end{cor}

\begin{proof}
By \cite{quadric-quartic}, such varieties are birationally superrigid, so the result follows from Theorem \ref{thm:superrigidity implies K-stability} and Lemma \ref{lem:lct-codim2-cpi}.
\end{proof}

Using the same strategy, we also prove the birational superrigidity of Fano complete intersections in large dimension.

\begin{lem} \label{lem:non-canonical to non-lc}
Let $(X,D)$ be a pair and $x\in X$. Assume that $X$ is smooth, $(X,D)$ has canonical singularities outside a subset of codimension at least $m+1$ in $X$, but is not canonical at $x$. Let $V\subseteq X$ be a general complete intersection subvariety of dimension $m$ containing $x$. Then $(V,D|_V)$ is not log canonical at $x$.
\end{lem}

\begin{proof} 
This is well known to experts but we include a proof for lack of a suitable reference. Let $Y\subseteq X$ be a general complete intersection subvariety of dimension $m+1$ and let $\Delta=D|_Y$. By \cite[Theorem 4.9(3)]{mmp} and our assumption, $(Y,\Delta)$ does not have canonical singularity at $x$ and there exists a prime divisor $E$ over $Y$ centered at $x$ such that $a(E;Y,\Delta)<0$. Let $V\subseteq Y$ be a general hypersurface containing $x$. Note that $\dim V=m$. Then $a(E;Y,\Delta+V) =  a(E;Y,\Delta)-\ord_E(V)<-1$ and in particular $(Y,\Delta+V)$ is not lc at $x$. By inversion of adjunction, $(V,\Delta|_V)$ is not lc at $x$ either.
\end{proof}

\begin{proof}[Proof of Theorem \ref{thm:superrigidity}]
Let $M\sim_\bQ -K_X$ be a movable boundary on $X$. We need to show that $(X,M)$ has canonical singularities. For this it suffices to consider movable boundaries of the form $M=a\cM$ where $\dim|\cM|=1$. With such $M$ we associate a codimension two cycle $M^2$ on $X$ by setting $M^2:=a^2(D_1\cdot D_2)$ where $D_1,D_2$ are two general members of $\cM$. Note that, by \cite[Proposition 2.1]{Suzuki-cpi}, we have $\mult_S(M^2)\le 1$ for every subvariety $S\subseteq X$ of dimension at least $2r$ (here and in what follows the multiplicities along subvarieties are taken in the sense of \cite[Section 4.3]{Fulton}); in other words, there exists a subset $Z\subseteq X$ of dimension at most $2r-1$ such that $\mult_x(M^2)\le 1$ for all $x\not\in Z$. Let $x\in X\backslash Z$ and let $S$ be a general surface section of $X$ containing $x$. By \cite[Theorem 0.1]{dFEM-mult-and-lct}, $(S,2M|_S)$ is lc at $x$ (note that as $M^2|_S$ is a complete intersection $0$-dimensional subscheme, its multiplicity at $x$ is the same as the Hilbert-Samuel multiplicity of its defining ideals), hence by inversion of adjunction, $(X,2M)$ is lc at $x$ as well. It follows that for all $0<\epsilon\ll 1$, the pair $(X,2(1-\epsilon)M)$ is klt outside $Z$. Let $x\in X$ be any point and let $Y\subseteq X$ be cut out by a general linear subspace $V\subseteq\bP^{n+r}$ of codimension $2r-1$ containing $x$. Then $Y\subseteq\bP^{n-r+1}$ is also a codimension $r$ complete intersection and we have $K_Y\sim 2(r-1)H$ where $H$ is the restriction of the hyperplane class. Let $D=2M|_Y$ and $L=2rH\sim_\bQ K_Y+D$. Since $V$ is general and $\dim Z \le 2r-1$, $(Y,(1-\epsilon)D)$ is klt outside a finite set of points (i.e. those in $V\cap Z$). Similar to the proof of Lemma \ref{lem:lct-cpi}, by Corollary \ref{cor:lct>1/2} we have $\lct(Y;D)>\frac{1}{2}$ as long as
\begin{equation} \label{eq:h^0(L)-rigidity}
    h^0(Y,L)\le h^0(\bP^{n-r+1},\cO_{\bP^{n-r+1}}(2r))=\binom{n+r+1}{2r}<\frac{(n-2r+1)^{n-2r+1}}{(n-2r+1)!}.
\end{equation}
Assuming this inequality for the moment, then $(Y,M|_Y)=(Y,\frac{1}{2}D)$ is klt. On the other hand by \cite[Proposition 2.1]{Suzuki-cpi}, we have $\mult_S(M)\le 1$ for every subvariety $S\subseteq X$ of dimension at least $r$, so $(X,M)$ is canonical outside a subset of dimension at most $r-1$ in $X$ by \cite[(3.14.1)]{Kol-sing-of-pairs}. Suppose that $(X,M)$ is not canonical at $x$. Then since $r-1<2r-1=\mathrm{codim}_X Y$, $(Y,M|_Y)$ is not lc by Lemma \ref{lem:non-canonical to non-lc}, which is a contradiction. Hence $(X,M)$ is canonical and we are done.

It remains to prove \eqref{eq:h^0(L)-rigidity} when $n\ge 10r$. Let $m=n-2r+1$; note that $m>8r$. As in the proof of Lemma \ref{lem:lct-cpi}, it is easy to see that \eqref{eq:h^0(L)-rigidity} is implied by the following weaker inequality
\[2^m\ge \left(\frac{e(m+3r)}{2r}\right)^{2r},
\]
or equivalently, $2^a\ge \frac{e^2}{4}(a+3)^2$ where $a=\frac{m}{r}$. This last inequality is obviously satisfied as $a>8$.
\end{proof}

\section{A criterion for K-stability} \label{sec:criterion}

In this section we give the proof of Theorem \ref{thm:criterion}.

\begin{defn} \label{defn:movable threshold}
Let $X$ be an $n$-dimensional variety and $L$ an ample divisor on $X$. Let $F$ be a prime divisor over $X$. The \emph{movable threshold} $\eta(L,F)$ (or simply $\eta(F)$) of $L$ with respect to $F$ is defined as the supremum of all $\eta>0$ such that every divisor in the stable base locus of $\pi^*L-\eta F$ is exceptional over $X$.
\end{defn}

Note that if $F$ is a dreamy divisor, then the supremum is indeed a maximum in the above definition.

\begin{lem} \label{lem:beta-inequality}
With notation as in Definitions \ref{defn:threshold and beta} and \ref{defn:movable threshold} and assuming that $X$ is $\bQ$-factorial and $\rho(X)=1$, we have the inequality
\[\frac{1}{(L^n)}\int_0^\infty \vol_X(L-xF)\mathrm{d}x \le \frac{1}{n+1}\tau(F)+\frac{n-1}{n+1}\eta(F).
\]
\end{lem}

\begin{rem}
Without the Picard number one assumption, Fujita in \cite[Proposition 2.1]{Fujita-plt-blowup} proves a weaker inequality where the right-hand side becomes $\frac{n}{n+1}\tau(F)$.
\end{rem}

\begin{proof}
The argument is a refinement of the proof of \cite[Proposition 2.1]{Fujita-plt-blowup} and \cite[Theorem 1.2]{rigid-imply-stable}, so we only indicate the difference. For ease of notation, let $\eta=\eta(F)$ and $\tau=\tau(F)$. Let $\pi:Y\to X$ be a projective birational morphism such that $F$ is a prime divisor on $Y$. Let 
\[b=\frac{1}{(L^n)}\int_0^\infty \vol_X(L-xF) \mathrm{d}x.
\]
As in the proof of \cite[Proposition 2.1]{Fujita-plt-blowup}, we have
\begin{equation} \label{eq:b}
    \int_0^{\tau} (x-b)\cdot \vol_{Y|F}(\pi^*L-xF) \mathrm{d}x = 0
\end{equation}
where $\vol_{Y|F}$ denotes the restricted volume of a divisor to $F$ (see \cite{ELMNP-restricted-volume}). For simplicity, we let $V_t=\vol_{Y|F}(\pi^*L-tF)$. It is clear (as in \cite[Proposition 2.1]{Fujita-plt-blowup}) that $F$ is not contained in the augmented base locus $\mathbf{B}_+(\pi^*L-xF)$ when $0\le x <\tau$. So by the log concavity property in \cite[Theorem A]{ELMNP-restricted-volume}, which holds for the restricted volume $\vol_{Y|F}(\pi^*L-xF)$ when $0\le x <\tau$, we have
\begin{equation} \label{eq:log-concavity}
    (x-x_0)\cdot V_x\le (x-x_0)\left(\frac{x}{x_0}\right)^{n-1}V_{x_0}
\end{equation}
for every $0\le x,x_0 \le \tau$. We may assume that $\eta<\tau$; otherwise, the lemma simply follows from \cite[Proposition 2.1]{Fujita-plt-blowup}. By the definition of pseudo-effective threshold, there exists an effective divisor $D\sim_\bQ -K_X$ such that $\ord_F(D)>\eta$. Since $X$ is $\bQ$-factorial and $\rho(X)=1$, we may assume that $D$ is irreducible. Such $D$ is necessarily unique by the definition of $\eta(F)$. In particular, there are no other effective divisors $D'\sim_\bQ -K_X$ with $\ord_F(D')>\ord_F(D)$ and hence $\ord_F(D)=\tau$. Moreover, if $D'\sim_\bQ -K_X$ is such that $\eta \le \ord_F(D')\le \tau$ and we write $D'=aD+M$ where $D\not\subseteq\Supp(M)$, then $\ord_F(M)\le\eta$. Let $\eta\le x \le\tau$. As \[\pi^*L-xF=\frac{\tau-x}{\tau-\eta}(\pi^*L-\eta F)+\frac{x-\eta}{\tau-\eta}(\pi^*L-\tau F),
\]
we see that $D$ appears in the stable base locus of $\pi^*L-xF$ with multiplicity $\ge \frac{x-\eta}{\tau-\eta}$ and we have the equality of restricted volumes
\begin{equation} \label{eq:restricted volume}
    V_x=\left(\frac{\tau-x}{\tau-\eta}\right)^{n-1}V_\eta.
\end{equation}
Now suppose first that $b\ge\eta$. Combining \eqref{eq:b}, \eqref{eq:log-concavity} (with $x_0=\eta$) and \eqref{eq:restricted volume} we have
\[0\le \int_0^\eta (x-b)\left(\frac{x}{\eta}\right)^{n-1}V_\eta \mathrm{d}x + \int_\eta^\tau (x-b)\left(\frac{\tau-x}{\tau-\eta}\right)^{n-1}V_\eta \mathrm{d}x,
\]
which reduces to $b\le \frac{1}{n+1}\tau+\frac{n-1}{n+1}\eta$. Suppose on the other hand that $b<\eta$. Then combining \eqref{eq:log-concavity} (with $x_0=b$ and $x=\eta$) and \eqref{eq:restricted volume} we have
\[V_x\le \left(\frac{\eta}{b}\right)^{n-1} \left(\frac{\tau-x}{\tau-\eta}\right)^{n-1}V_b
\]
when $\eta\le x \le\tau$. Combining this with \eqref{eq:log-concavity} (with $x_0=b$ again) and \eqref{eq:b} we have
\[0\le \int_0^\eta (x-b)\left(\frac{x}{b}\right)^{n-1}V_b \mathrm{d}x + \int_\eta^\tau (x-b)\left(\frac{\eta}{b}\right)^{n-1}\left(\frac{\tau-x}{\tau-\eta}\right)^{n-1}V_b \mathrm{d}x,
\]
which again reduces to $b\le \frac{1}{n+1}\tau+\frac{n-1}{n+1}\eta$. This proves the lemma.
\end{proof}

Comparing with the expression of $\beta(F)$ we immediately obtain the following.

\begin{cor} \label{cor:beta>0}
Assume that $\frac{1}{n+1}\tau(F)+\frac{n-1}{n+1}\eta(F)\le A_X(F)$ $($resp. $<A_X(F))$. Then $\beta(F)\ge 0$ $($resp. $>0)$. \qed
\end{cor}

\begin{proof}[Proof of Theorem \ref{thm:criterion}]
Let $F$ be a dreamy divisor over $X$, let $\eta=\eta(F)$, and let $\tau=\tau(F)$. Then for $m\gg0$, the linear system $|-mK_X-m\tau F|$ (i.e., the sublinear system of $|-mK_X|$ consisting of divisors that vanish with order at least $m\tau$ along $F$) is nonempty, whereas $|-mK_X-m\eta F|$ is movable. Let $D\in|-mK_X-m\tau F|$ and $M=\frac{1}{m}|-mK_X-m\eta F|$. Then $M$ is a movable boundary, $D\sim_\bQ M\sim_\bQ -K_X$, $\ord_F(D)=\tau$ and $\ord_F(M)=\eta$. Thus if $(X,\frac{1}{n+1}D+\frac{n-1}{n+1}M)$ is lc (resp. klt), then we have $\frac{1}{n+1}\tau+\frac{n-1}{n+1}\eta\le A_X(F)$ (resp. $<A_X(F)$). As this holds for every dreamy divisor $F$, $X$ is K-semistable (resp. K-stable) by Theorem \ref{thm:beta-criterion} and Corollary \ref{cor:beta>0}.
\end{proof}

\section{Intersection of quadric and cubic} \label{sec:X_2,3}

In this section, we make a more delicate use of Theorem \ref{thm:lct-estimate} to prove the K-stability of $X_{2,3}\subseteq \bP^5$ (based on the criterion given by Theorem \ref{thm:criterion}). Again, we start with some lower bound of $\sigma_{n,\lambda}$ and $\bar{\sigma}_{n,\lambda}$. In the surface case, these numbers can be approximated quite precisely using Pick's theorem. 

\begin{lem} \label{lem:sigma_2,m}
Let $m\in\bZ_+$. Then $\sigma_{2,m}\ge \frac{1}{2}(4m^2+3m+3)$.
\end{lem}

\begin{proof}
Let $\mathbf{a}=(\frac{1}{s},\frac{1}{t})$ be such that $(m,m)\in Q=Q_\mathbf{a}$. Then we have $\frac{1}{s}+\frac{1}{t}<\frac{1}{m}$ and $s+t>4m$. We may slightly decrease $s,t$ and assume that $s,t\not\in\bZ$. Let $u=\lfloor s \rfloor$ and let $v=\lfloor t \rfloor$. Then $u+v\ge 4m-1$. Now consider the polygon $P$ given by the following vertices: $(0,0)$, $(u,0)$, $(m,m)$ and $(0,v)$. Clearly $P\subseteq Q$, so it suffices to prove
\begin{equation} \label{eq:P cap Z^2}
    \#(P\cap\bZ^2)\ge \frac{1}{2}(4m^2+3m+3).
\end{equation}
On the other hand, by Pick's theorem, we have $i+\frac{1}{2} b=A+1$ where $i=\#(P^\circ\cap\bZ^2)$, $b=\#(\partial P\cap\bZ^2)\ge u+v+2$ and $A=\mathrm{Area}(P)=\frac{1}{2}m(u+v)$, hence 
\[\#(P\cap\bZ^2)=i+b=A+\frac{1}{2}b+1\ge \frac{1}{2}(m+1)(u+v)+2 \ge \frac{1}{2}(m+1)(4m-1)+2,
\]
which gives \eqref{eq:P cap Z^2} and we are done.
\end{proof}

\begin{lem} \label{lem:sigma-bar_2,m}
Let $m\in\bZ_+$. Then $\bar{\sigma}_{2,m}\ge m(2m+1)$.
\end{lem}

\begin{proof}
Let $s,t,Q$ be as in the proof of Lemma \ref{lem:sigma_2,m}. It suffices to show that $\#(Q\cap\bZ^2)\ge m(2m+1)$. We have $\frac{1}{s}+\frac{1}{t}\le\frac{1}{m}$ and thus $s+t\ge 4m$. Let $0<\epsilon\ll 1$ and let $u=\lfloor s-\epsilon \rfloor$, $v=\lfloor t-\epsilon \rfloor$. Then $u+v\ge s+t-2\ge 4m-2$. Consider again the polygon $P$ given by the vertices  $(0,0)$, $(u,0)$, $(m,m)$ and $(0,v)$. As before by Pick's theorem we have
\[\#(P\cap\bZ^2)=i+b=A+\frac{1}{2}b+1\ge \frac{1}{2}(m+1)(u+v)+2 \ge (m+1)(2m-1)+2=m(2m+1)+1.
\]
Since every lattice point of $P$ (except possibly $(m,m)$) is contained in $Q$, we obtain $\#(Q\cap\bZ^2)\ge m(2m+1)$ as desired.
\end{proof}

We are now ready to prove the following.

\begin{prop} \label{prop:X_2,3}
The smooth complete intersection $X=X_{2,3}\subseteq \bP^5$ of a quadric and a cubic is K-stable.
\end{prop}

\begin{proof}
Let $D\sim_\bQ -K_X$ be an effective divisor on $X$ and let $M\sim_\bQ -K_X$ be a movable boundary. By Theorem \ref{thm:criterion}, it suffices to show that $(X,\frac{1}{4}D+\frac{1}{2}M)$ is klt (note that $n=3$). Since being klt is preserved under convex linear combination and $\rho(X)=1$, we may assume that $D$ is irreducible. As $\Pic(X)$ is generated by $-K_X$, we have $D=\frac{1}{r}D_0$ where $D_0$ is integral and $D_0\in |-rK_X|$ for some $r\in\bZ$. Let $H$ be the hyperplane class on $X$ and let  $\Delta=\frac{1}{4}D+\frac{1}{2}M$. Depending on the value of $r$, we separate into three cases.

(1) First suppose that $r\ge 3$. Then for $0<\epsilon\ll 1$, $(X,12(1-\epsilon)\Delta)$ is klt outside a subset of dimension at most $1$ since every component of $12\Delta$ has coefficient at most $1$. Let $x\in X$ and let $S\subseteq X$ be a general hyperplane section containing $x$. Let $\Delta_S=12\Delta|_S$. Then $(S,(1-\epsilon)\Delta_S)$ is klt outside a finite number of points. By adjunction $S$ is a smooth K3 surface and $K_S+\Delta_S\sim_\bQ 9H|_S$. We claim that $\lct(S,\Delta_S)>\frac{1}{12}$. Indeed, by Corollary \ref{cor:lct>1/(1+lambda)}, it suffices to show that $h^0(S,\cO_S(9H))<\bar{\sigma}_{2,11}$. But by Riemann-Roch, we have $h^0(S,\cO_S(9H))=\frac{9^2}{2}(H|_S^2)+2=245$ while by Lemma \ref{lem:sigma-bar_2,m} with $m=11$ we have $\bar{\sigma}_{2,11}
\ge 11\cdot 23=253$, proving the claim. It follows that $(S,\frac{1}{12}\Delta_S)=(S,\Delta|_S)$ is klt, hence by inversion of adjunction, $(X,\Delta)$ is also klt at $x$. Since $x\in X$ is arbitrary, we see that $(X,\Delta)$ is klt in this case.

(2) Next suppose that $r=2$. Let $\Gamma=\frac{4}{3}\Delta=\frac{1}{3}D+\frac{2}{3}M$.

\begin{claim*}
The pair $(X,\Gamma)$ is log canonical in dimension $1$.
\end{claim*}

\begin{proof}[Proof of Claim]
Suppose not, and let $C$ be a curve in the non-lc locus of $(X,\Gamma)$. We first show that $C$ is a line. Otherwise if $S$ is a general hyperplane section then by adjunction $(S,\Gamma|_S)$ is not log canonical at at least $2$ points (those in $C\cap S$), say, $x_1$ and $x_2$. Let $\Delta_S=6\Gamma|_S$. As before $(S,(1-\epsilon)\Delta_S)$ is klt outside a finite set of points for $0<\epsilon\ll 1$ and we have $K_S+\Delta_S\sim_\bQ 6H|_S$. As $(S,\Gamma|_S)$ is not lc at $x_i$ $(i=1,2)$, we have $\lct(S;\Delta_S)<\frac{1}{6}$ in the neighbourhood of $x_i$, thus by Theorem \ref{thm:lct-estimate} and Remark \ref{rem:pointwise-estimate} we have $\lct(S;\Sigma_i)<\frac{1}{5}$ where $\Sigma_i$ is a $0$-dimensional subscheme supported at $x_i$ and $\ell(\cO_{\Sigma_1})+\ell(\cO_{\Sigma_2})\le h^0(S,\cO_S(6H))=\frac{6^2}{2}(H|_S^2)+2=110$. But by Lemma \ref{lem:sigma_n and lct} and \ref{lem:sigma_2,m} with $m=5$ we also have $\ell(\cO_{\Sigma_i})\ge \sigma_{2,5} \ge 59$ $(i=1,2)$, a contradiction. Hence $\deg C\le 1$ and $C$ is a line. 

We next prove that $\mult_C D\le 1$, or equivalently, $s:=\mult_C D_0\le 2$. To see this, take a general hyperplane section $S$ containing the line $C$. By dimension count it is not hard to see that $S$ is smooth (by Bertini's theorem, the singular locus of $S$ is contained in $C$, thus if $S$ is singular, it is the tangent hyperplane section of some $x\in C$; but these tangent hyperplanes only vary in a 2-dimensional family, whereas there is a 3-dimensional family of hyperplanes containing $C$). We have $D_0|_S=sC+Z$ where $Z$ is integral. As $S$ is a  K3 surface and $C\cong\bP^1$, we have $(H|_S^2)=6$, $(H\cdot C)=1$, $(C^2)=-2$ and hence $(Z^2)=(2H|_S-sC)^2=24-4s-2s^2$. On the other hand, since $Z$ is an integral curve on a K3 surface we have $(Z^2)\ge -2$. Thus $24-4s-2s^2\ge -2$ and it follows that $s\le 2$ as $s\in\bZ$.

Now if $(X,\Gamma)$ is not lc along $C$, then by \cite[Theorem 2.2]{dFEM-mult-and-lct} we have 
\[
\mult_C(M^2)>9-3\mult_C D \ge 6.
\]
But $\mult_C(M^2)\le \deg(M^2)=(M^2\cdot H)=6$, a contradiction. This proves the claim.
\end{proof}

It follows from the claim that $(X,(1-\epsilon)\Gamma)$ is klt outside a finite set of points. Note that $K_X+\Gamma\sim_\bQ 0$. We may apply Theorem \ref{thm:lct-estimate} with $L=0$ and $\lambda=n$ (as in the proof of Corollary \ref{cor:hypersurface lct}) to conclude that $\lct(X;\Gamma)\ge\frac{3}{4}$, with equality if and only if $\mult_x \Gamma=4$ for some $x\in X$. But it is easy to see that $\mult_x D\le (D\cdot H^2)=6$ and $(\mult_x M)^2\le (M^2\cdot H)=6$, thus $\mult_x \Gamma=\frac{1}{3}\mult_x D + \frac{2}{3}\mult_x M\le 2+\frac{2}{3}\sqrt{6}<4$. Therefore the equality of $\lct(X;\Gamma)\ge\frac{3}{4}$ in never achieved and $(X,\Delta)=(X,\frac{3}{4}\Gamma)$ is klt as desired.

(3) We are left with the case in which $r=1$, in other words, $D$ is a hyperplane section. Let $\Gamma=\frac{4}{3}\Delta$ be as in the previous case. Again we claim the following.

\begin{claim*}
The pair $(X,\Gamma)$ is log canonical in dimension $1$.
\end{claim*}

\begin{proof}[Proof of Claim]
The proof is very similar to the previous case, so we only give a sketch. Let $C$ be a curve in the non-lc locus of $(X,\Gamma)$. We have $\deg C\le 2$ (otherwise if $S$ is a general hyperplane section and $\Delta_S=3\Gamma|_S$, then lc centers of $(S,(1-\epsilon)\Delta_S)$ are isolated points and $\lct(S;\Delta_S)<\frac{1}{3}$ in the neighbourhood of at least $3$ points. Hence there exists a $0$-dimensional subscheme $\Sigma$ supported on one of these points such that $\ell(\cO_\Sigma)\le\lfloor \frac{1}{3} h^0(S,\cO_S(3H)) \rfloor =9$ and $\lct(S;\Sigma)<\frac{1}{2}$, but the latter inequality implies that $\ell(\cO_\Sigma)\ge \sigma_{2,2}\ge 13$ by Lemma \ref{lem:sigma_n and lct} and \ref{lem:sigma_2,m} with $m=2$, a contradiction). If $C$ is a line, we simply argue as in the previous case (i.e. take a general hyperplane section containing $C$ to prove $\mult_C D\le 1$ and then apply \cite[Theorem 2.2]{dFEM-mult-and-lct} to get a contradiction). So we assume that $C$ is a conic. We claim that $\mult_C D\le 2$. Suppose not. Then as $D$ is an integral divisor (recall that $r=1$), we have $\mult_C D\ge 3$. Let $S$ be a general hyperplane section containing $C$. Then $S$ is smooth along $C$: otherwise, there exists $x\in C$ such that $\mult_x S\ge 2$; as $\mult_x D\ge 3$, $D\cap S$ is a curve with degree $6$ and multiplicity at least $6$ at $x$, hence it is a union of $6$ lines; but $D\cap S$ already contains the conic, a contradiction. As $6\ge\deg(D\cdot S)\ge \mult_C D \cdot \deg C\ge 6$, we must have $D|_S=3C$. Since $D$ is a hyperplane section we have $(C^2)>0$ on $S$; but as $C\cong\bP^1$ is in the smooth locus of the (possibly singular) surface $S$ and $S$ has trivial canonical line bundle, we have $(C^2)=-2$ by adjunction, a contradiction. Hence we always have $\mult_C D\le 2$. Now another application of \cite[Theorem 2.2]{dFEM-mult-and-lct} gives
\[
\mult_C(M^2)>9-3\mult_C D \ge 3
\]
and therefore $\deg(M^2)=(M^2\cdot H)\ge \mult_C(M^2)\cdot \deg C>6$, a contradiction. This proves the claim.
\end{proof}

We are in the same situation as in the previous case and the rest of the proof is identical to the one there. Hence in all cases $(X,\Delta)$ is klt and we conclude that $X$ is K-stable.
\end{proof}

Theorem \ref{thm:K-stability of cpi} now follows by combining Proposition \ref{prop:X_2,3}, Theorem \ref{thm:superrigidity implies K-stability}, Theorem \ref{thm:superrigidity} and Lemma \ref{lem:lct-cpi}.

\appendix

\section{Conditional birational superrigidity}

\centerline{Charlie Stibitz \footnote{CS would like to thank his advisor J\'anos Koll\'ar for  constant support and he also wishes to thank Fumiaki Suzuki for helpful conversation.}, Ziquan Zhuang}

\bigskip

In an attempt to study the birational geometry of Fano varieties, Suzuki proposed (in a paper that was later withdrawn) the following notion of conditional birational superrigidity:

\begin{defn}
Let $X$ be a Fano manifold of Picard number one and let $s\ge 2$ be an integer. Consider the following condition on $X$:
    \begin{enumerate}
        \item[$(C_{s})$] every birational map from $X$ to a Mori fiber space whose undefined locus has codimension at least $s$ is an isomorphism.
    \end{enumerate}
We say $X$ is \emph{conditionally birationally superrigid} if it satisfies condition $(C_{i_X+1})$ where $i_X$ is the index of $X$ (i.e. $-K_X=i_X H$ where $H$ is the ample generator of $\Pic(X)$).
\end{defn}

For example, when $X$ has index one, conditional birational superrigidity is just the usual birational superrigidity. On the other hand, if $X$ is a complete intersection of index $i_X\ge 2$, then $X$ does not satisfy condition $(C_{i_X})$ due to the existence of general linear projections $X\dashrightarrow \bP^{i_X-1}$.

In this appendix, we apply the main technique of this paper to give a short proof of the conditional birational superrigidity of Fano complete intersections in large dimension. Indeed, we prove something stronger.

\begin{thm} \label{thm:(X,M) canonical}
Let $m,r\in\bZ_+$. Then there exists an integer $N=N(r,m)$ depending only on $m$ and $r$ such that for every smooth Fano complete intersection of codimension $r$ and dimension $n\ge N$ in $\bP^{n+r}$ and every movable boundary $M\sim_\bQ mH$ whose base locus has codimension at least $m+1$ $($where $H$ is the hyperplane class$)$, the pair $(X,M)$ is canonical.
\end{thm}

\begin{cor} \label{cor:conditional-superrigid}
Let $r,s\in\bZ_+$. Then there exists an integer $N=N(r,s)$ depending only on $r$ and $s$ such that every smooth Fano complete intersection of index $s$ and codimension $r$ in $\bP^{n+r}$ is conditionally birationally superrigid if $n\ge N$.
\end{cor}

\begin{rem}
It is also conjectured that for every birational map $\phi:X\dashrightarrow X'$ from a Fano hypersurface of index $s$ to a Mori fiber space $f:X'\rightarrow S$ that is not an isomorphism, we have $\dim S\le s-1$, see e.g. \cite[Conjecture 1.1]{P-index-two}.
\end{rem}

As before, we need some estimate of $\sigma_{n,\lambda}$ for the proof of Theorem \ref{thm:(X,M) canonical}. This is given as follows:

\begin{lem} \label{lem:sigma_n when n>>0}
Fix $\lambda>0$. Then there exists a constant $c>1$ $($depending on $\lambda)$ such that $\sigma_{n,\lambda}> c^n$ for $n\gg 0$.
\end{lem}

\begin{proof}
We may assume that $\lambda<1$ since $\sigma_{n,\lambda}$ is non-decreasing in the variable $\lambda$. Let $\mathbf{a}\in \bR^n_+$ be such that $(\lambda,\cdots,\lambda)\in Q=Q_{\mathbf{a}}$. We may assume that $\mathbf{a}=(a_1,\cdots,a_n)$ where $a_1\le \cdots \le a_n$. We have $a_1+\cdots+a_n<\frac{1}{\lambda}$. Let $m=\lfloor n\lambda \rfloor$. Then $a_1+\cdots+a_m\le \frac{m}{n}(a_1+\cdots+a_n)<\frac{m}{n\lambda}\le 1$. It follows that every vertex of $[0,1]^m\times \{(0,\cdots,0)\}$ is contained in $Q$, hence $\#(Q\cap\bZ^n)\ge 2^m > 2^{n\lambda-1}$ and therefore the statement of the lemma holds for any $1<c<2^{\lambda}$.
\end{proof}

\begin{proof}[Proof of Theorem \ref{thm:(X,M) canonical}]

By \cite[Proposition 2.1]{Suzuki-cpi}, we have $\mult_S(M^m)\le m^m$ for every subvariety $S\subseteq X$ of dimension at least $mr$; in other words, there exists a subset $Z\subseteq X$ of dimension at most $mr-1$ such that $\mult_x(M^m)\le m^m$ for every $x\not\in Z$. We first claim that $(X,M)$ has canonical singularities outside $Z$. Suppose this is not the case and $(X,M)$ is not canonical at $x\not\in Z$. Let $V\subseteq X$ be a general complete intersection subvariety of dimension $m$ containing $x$. Then since $(X,M)$ is obviously canonical outside the base locus of $M$, which has codimension at least $m+1$, we see that $(V,M|_V)$ is not lc at $x$ by Lemma \ref{lem:non-canonical to non-lc}. But since $V$ is general we have $\mult_x(M|_V^m)=\mult_x(M^m)\le m^m$ (see e.g. \cite[Proposition 4.5]{dFEM-bounds-on-lct}) and since $M|_V^m$ is a $0$-dimensional complete intersection subscheme, its multiplicity is the same as the Hilbert-Samuel multiplicity of its defining ideal, so by \cite[Theorem 0.1]{dFEM-mult-and-lct}, $(V,M|_V)$ is lc at $x$, a contradiction. This proves the claim.

By \cite[Proposition 2.1]{Suzuki-cpi} again, we have $\mult_S(M^{m+1})\le m^{m+1}$ for every subvariety $S\subseteq X$ of dimension at least $(m+1)r$, hence by a similar application of \cite[Theorem 0.1]{dFEM-mult-and-lct} as before, the pair $(X,\frac{m+1}{m} M)$ is log canonical outside a subset of dimension at most $mr+r-1$. Let $x\in X$ be an arbitrary point and let $Y\subseteq X$ be a general linear space section of codimension $mr+r-1$ containing $x$. Then the pair $(Y,\frac{m+1}{m} M|_Y)$ is log canonical outside a finite set of points. Let $L=(mr+m+r-1)H$. Then since $X$ is Fano, $L-(K_Y+\frac{m+1}{m}M)$ is nef, hence by Corollary \ref{cor:lct>1/(1+lambda)} (with $\lambda=m^{-1}$) we see that $\lct(Y,\frac{m+1}{m} M|_Y)\ge \frac{m}{m+1}$ as long as
\begin{equation} \label{eq:h^0(L)-conditional}
    h^0(Y,L)\le h^0(\bP^{n-mr+1},\cO_{\bP^{n-mr+1}}(mr+m+r-1))=\binom{n+m+r}{mr+m+r-1}< \sigma_{n,m^{-1}}
\end{equation}
By Lemma \ref{lem:sigma_n when n>>0}, $\sigma_{n,m^{-1}}$ grows exponentially with $n$, hence \eqref{eq:h^0(L)-conditional} is always satisfied for $n\ge N$ where $N$ is an integer depending only on $m$ and $r$. It follows that $(Y,M|_Y)$ is log canonical when $n\ge N$. On the other hand, $(X,M)$ is canonical outside $Z$, which has codimension at least $n-mr+1>\dim Y$ in $X$, thus by Lemma \ref{lem:non-canonical to non-lc}, $(X,M)$ is also canonical at $x$. Since $x\in X$ is arbitrary, we are done.
\end{proof}

\begin{rem}
For any given $m$ and $r$, we can always find an explicit $N(r,m)$ using the inequality \eqref{eq:h^0(L)-conditional} and the estimate $\sigma_{n,m^{-1}}>2^{\frac{n}{m}-1}$ from the proof of Lemma \ref{lem:sigma_n when n>>0}. For example, we may take $N(1,2)=36$ and $N(1,4)=200$.
\end{rem}

\begin{proof}[Proof of Corollary \ref{cor:conditional-superrigid}]
Let $N$ be the number given by Theorem \ref{thm:(X,M) canonical} with $m=s$. Let $X\subseteq\bP^{n+r}$ be a smooth Fano complete intersection of index $s$, codimension $r$ and dimension $n$. Suppose that $\phi:X\dashrightarrow X'$ is a birational map from $X$ to a Mori fiber space $X'$ such that $\phi$ is not an isomorphism and the undefined locus of $\phi$ has codimension at least $s+1$. By the usual method of maximal singularities (see e.g. \cite[Section 2]{P-book}), we find a movable boundary $M\sim_\bQ -K_X = sH$ whose base locus is contained in the undefined locus of $\phi$ (in particular, the base locus has codimension at least $s+1$) such that the pair $(X,M)$ is not canonical. By Theorem \ref{thm:(X,M) canonical}, this is impossible if $n\ge N$.
\end{proof}

\bibliography{ref}

\end{document}